\documentclass[12pt,a4paper]{article}
\usepackage{amsmath,amssymb}

\newtheorem{theorem}{Theorem}[section]
\newtheorem{proposition}{Proposition}[section]

\newenvironment{proof}{\noindent Proof:}{$\Box$}

\newcommand{\N}{{\mathbb N}}

\newcommand{\R}{{\mathbb R}}
\newcommand{\C}{{\mathbb C}}

\newcommand{\Ann}{\mbox{{\rm Ann}}}
\newcommand{\myRe}{{\rm Re\,}\,}
\newcommand{\Fsc}{{\mathcal F}}

\newcommand{\Dsc}{{\mathcal D}}

\newcommand{\Db}{{\mathcal D}'_M}

\title{Annihilators of Laurent coefficients of the complex power for 
normal crossing singularity}
\author{Toshinori Oaku}
\date{}
\begin{document}
\maketitle

\begin{abstract}  
Let $f$ be a real-valued real analytic function defined on an open set 
of $\R^n$. 
Then the complex power $f_+^\lambda$ is defined as a distribution with 
a holomorphic parameter $\lambda$. We determine the annihilator
(in the ring of differential operators) of 
each coefficient of the principal part of 
the Laurent expansion of $f_+^\lambda$ about 
$\lambda=-1$ in case $f=0$ has a normal crossing singularity. 
\end{abstract}

\section{Introduction}

Let $\Dsc_X$ be the sheaf of linear differential operators 
with holomorphic coefficients on the $n$-dimensional complex 
affine space $X = \C^n$. 
We denote by $\Dsc_M$ the sheaf theoretic restriction of 
$\Dsc_X$ to the $n$-dimensional real affine space $M = \R^n$, 
which is the sheaf of linear differential operators 
whose coefficients are
complex-valued real analytic functions. 
Let us denote by $\Dsc_0 = (\Dsc_M)_0$, for the sake of brevity, 
the stalk of $\Dsc_M$ (or of $\Dsc_X$) at
the origin $0 \in M$, which is a (left and right) Noetherian ring. 

Let $\Db$ be the sheaf on $M$ of the distributions  
(generalized functions) in the sense of 
L.\ Schwartz.  
In general, for a sheaf $\Fsc$ on $M$ and an open subset $U$ of $M$, 
we denote by $\Gamma(U,\Fsc) = \Fsc(U)$ 
the set of the sections of $\Fsc$ on $U$. 
Let $C^\infty_0(U)$ be the set of the complex-valued $C^\infty$ functions 
defined on $U$ whose support is a compact set contained in $U$. 
Then 
$\Gamma(U,\Db)$ consists of the $\C$-linear maps 
\[
u : C^\infty_0(U) \ni \varphi \longmapsto \langle u,\varphi\rangle \in  \C
\]
which are continuous in the sense that 
$\lim_{j\rightarrow\infty}\langle u,\varphi_j\rangle = 0$ holds for 
any sequence $\{\varphi_j\}$ of $C^\infty_0(U)$ 
if there is a compact set $K \subset U$ such that $\varphi_j = 0$ 
on $U \setminus K$ and 
\[
 \lim_{j\rightarrow\infty}\sup_{x\in U} |\partial^\alpha\varphi_j(x)| = 0
\quad\mbox{for any $\alpha \in \N^n$},
\]
where we use the notation $x = (x_1,\dots,x_n)$, 
$\N = \{0,1,2,\dots\}$ and 
$\partial^\alpha = \partial_1^{\alpha_1}\cdots \partial_n^{\alpha_n}$ 
with $\partial_j = \partial/\partial x_j$. 

For a distribution $u$ defined on an open set $U$ of $M$, 
its annihilator $\Ann_{\Dsc_M} u$ 
in $\Dsc_M$ is defined to be the sheaf 
of left ideals of sections $P$ of  $\Dsc_M$ which annihilate $u$. 
That is, for each open subset $V$ of $U$, we have by definition 
\[
\Gamma(V,\Ann_{\Dsc_M}u) = \{P \in \Dsc_M(V) \mid Pu = 0 \mbox{ on } V\}.
\]
Its stalk $\Ann_{\Dsc_0} u$ at $0 \in M$ is a left ideal of $\Dsc_0$. 

Now let $f$ be a real-valued real analytic function defined on an 
open set $U$ of $M$. 
Then for a complex number $\lambda$ with non-negative real part 
($\myRe \lambda \geq 0$), the distribution $f_+^\lambda$ is defined 
to be the locally integrable function
\[
 f_+^\lambda(x) := \left\{\begin{array}{ll}
 f(x)^\lambda = \exp(\lambda\log f(x)) & \mbox{if $f(x) > 0$}\\
 0 & \mbox{if $f(x) \leq 0$}
\end{array}\right.
\]
on $U$ 
and is holomorphic with respect to $\lambda$ for $\myRe \lambda > 0$.

For each $x_0 \in U$, there exist a nonzero polynomial $b_{f,x_0}(s)$ 
in an indeterminate $s$ and some $P(s) \in (\Dsc_{M})_{x_0}[s]$ such that
\[
 b_{f,x_0}(\lambda)f_+^\lambda = P(\lambda)f_+^{\lambda+1}
\]
holds in a neighborhood of $x_0$ for $\myRe \lambda > 0$. 
It follows that $f_+^\lambda$ is a distribution-valued meromorphic 
function on the whole complex plane $\C$ with respect to $\lambda$. 
This is called the complex power, 
and for a compactly supported $C^\infty$-function $\varphi$ on $U$, 
the meromorphic function 
$\langle f_+^\lambda, \varphi\rangle$ 
in $\lambda$ is called the local zeta function
(see, e.g., \cite{I}). 

By virtue of Kashiwara's theorem on the rationality of $b$-functions 
(\cite{K}), 
the poles of $f_+^\lambda$ are negative rational numbers. 
Let $\lambda_0$ be a pole of $f_+^\lambda$ and $x_0$ be a point of $U$.  
Then there exist a positive integer $m$, an open neighborhood $V$ 
of $x_0$, an open neighborhood $W$ of $\lambda_0$ in $\C$, 
and distributions $u_k$ defined on $V$ such that
\[
f_+^\lambda = u_{-m}(\lambda-\lambda_0)^{-m}
+ 
\cdots + u_{-1}(\lambda-\lambda_0)^{-1} + u_0 + u_1(\lambda-\lambda_0) 
+ \cdots
\]
holds as distribution on $V$ for any 
$\lambda \in W \setminus \{\lambda_0 \}$. 
To determine the poles of $f_+^\lambda$, and its 
Laurent expansion at each pole is an interesting problem and
has been investigated by many authors. 

From the viewpoint of $D$-module theory, it would be interesting if 
we can compute the annihilator of each Laurent coefficient as above 
explicitly. 
For example, we compared the annihilator of the residue of $f_+^\lambda$ 
at $\lambda = -1$ with that of local cohomology group supported 
on $f=0$ in \cite{O}. 

In this paper, we treat the case where $f=0$ has 
a normal crossing singularity at the origin and 
determine the annihilators of the coefficients of the negative degree 
part of the Laurent expansion  
about $\lambda = -1$. 
The two dimensional case was treated in \cite{O}.

\section{Main results}

Let $x = (x_1,\dots,x_n)$ be the coordinate of $M = \R^n$. 

\begin{proposition}\label{prop:main}
The distribution $(x_1\cdots x_n)_+^\lambda$ has a pole of order $n$ 
at $\lambda = -1$.
Let
\[
(x_1\cdots x_n)_+^\lambda = 
\sum_{j=-n}^\infty (\lambda+1)^ju_j
\]
be the Laurent expansion of the distribution $(x_1\cdots x_n)_+^\lambda$ 
with respect to the holomorphic parameter $\lambda$ about 
$\lambda = -1$, with $u_j \in \Db(M)$ for $j \geq -n$. 
Then for $k=0,1,\dots,n-1$, 
the left ideal
$\Ann_{\Dsc_0} u_{-n+k}$ of $\Dsc_0$ is generated by 
\[
x_{j_1}\cdots x_{j_{k+1}} \quad (1 \leq j_1 < \cdots < j_{k+1} \leq n), 
\quad
x_1\partial_1 - x_i\partial_i \quad (2 \leq i \leq n). 
\]
\end{proposition}

\begin{proof}
In one variable $t$, we have
\begin{align*}
t^\lambda_+ &= 
(\lambda+1)^{-1}\partial_t t_+^{\lambda+1}
\\&=
(\lambda+1)^{-1}\partial_t
\left\{Y(t) + 
\sum_{j=1}^\infty \frac{1}{j!}(\lambda+1)^j(\log t_+)^j
\right\}
\\&=
(\lambda+1)^{-1}\delta(t) + 
\sum_{j=1}^\infty \frac{1}{j!}(\lambda+1)^{j-1}\partial_t(\log t_+)^j,
\end{align*}
where $(\log t_+)^j$ is the distribution defined by the pairing
\[
\langle (\log t_+)^j,\,\varphi\rangle 
= \int_0^\infty (\log t)^j\varphi(t)\,dt
\]
for $\varphi \in C^\infty_0(\R)$. 

Let us introduce the following notation: 
\begin{itemize}
\item
For a nonnegative integer $j$, we set
\[
 h_j(t) = \left\{ \begin{array}{ll} 
\delta(t)  & (j=0), \\
\frac{1}{j!}\partial_t(\log t_+)^j \quad & ( j \geq 1) 
\end{array}\right.
\]
with $\partial_t = \partial/\partial_t$ and 
\[
h_\alpha(x) = h_{\alpha_1}(x_1)\cdots h_{\alpha_n}(x_n)
\] 
for a multi-index $\alpha = (\alpha_1,\dots,\alpha_n) \in \N^n$. 
\item
For a multi-index $\alpha = (\alpha_1,\dots,\alpha_n) \in \N^n$, 
we set 
\[
|\alpha| = \alpha_1 + \cdots + \alpha_n,\quad
[\alpha] = \max\{ \alpha_i \mid 1 \leq i \leq n\}.
\]
\item 
Set 
$S(n) = \{\sigma = (\sigma_1,\dots,\sigma_n) \in \{1,-1\}^n \mid 
\sigma_1\cdots \sigma_n = 1\}$. 
\end{itemize}

Since
\[
(x_1\cdots x_n)_+^\lambda
= \sum_{\sigma \in S(n)}
(\sigma_1 x_1)_+^\lambda\cdots(\sigma_n x_n)_+^\lambda, 
\]
we have
\begin{align*}
u_{-n+k}(x) &= \sum_{\sigma\in S(n)}\sum_{|\alpha|=k} 
h_\alpha(\sigma x). 
\end{align*}

In particular, we have
\[
u_{-n}(x) = \sum_{\sigma\in S(n)} \delta(\sigma_1x_1)\cdots\delta(\sigma_nx_n)
= 2^{n-1}\delta(x_1)\cdots\delta(x_n).
\]
It follows that $\Ann_{\Dsc_0}u_{-n}$ is generated by 
$x_1,\dots,x_n$. This proves the assertion for $k=0$ 
since $x_1\partial_1 - x_i\partial_i = \partial_1x_1 - \partial_ix_i$ 
belongs to the left ideal of $\Dsc_0$ generated by $x_1,\dots,x_n$. 

We shall prove the assertion by induction on $k$. 
Assume $k \geq 1$ and $P \in \Dsc_0$ annihilates 
$u_{-n+k}$, that is, $Pu_{-n+k}=0$ holds on a neighborhood of $0 \in M$.
By division, there exist $Q_1,\dots,Q_r,R \in \Dsc_0$ such that
\begin{align}
P &= Q_1\partial_1x_1 + \cdots + Q_n\partial_nx_n + R,
\label{eq:Pdivision}
\\
R &= \sum_{\alpha_1\beta_1 = \cdots =\alpha_n\beta_n = 0} 
a_{\alpha,\beta} x^\alpha\partial^\beta
\qquad (a_{\alpha,\beta} \in \C).
\nonumber
\end{align}
Since 
\begin{equation}\label{eq:u_-n+k}
u_{-n+k}(x)
= \sum_{\sigma\in S(n)}\sum_{|\alpha|=k,\,[\alpha]=1}  h_\alpha(\sigma x)
+ \sum_{\sigma\in S(n)}\sum_{|\alpha|=k,\,[\alpha] \geq 2} 
h_\alpha(\sigma x), 
\end{equation}
we have
\begin{align*}
u_{-n+k}(x)
&= 2^{n-k-1}\delta(x_{1})\cdots\delta(x_{n-k})h_1(x_{n-k+1})\cdots h_1(x_n)
\\&
= 2^{n-k-1}\delta(x_{1})\cdots\delta(x_{n-k})
\frac{1}{x_{n-k+1}}\cdots \frac{1}{x_n}
\end{align*}
on the domain $x_{n-k+1} > 0,\dots,x_{n} > 0$. 
Note that $\partial_ix_i$ annihilates both $\delta(x_i)$ 
and $x_i^{-1}$. 
Hence
\begin{align*}
0 &= Pu_{-n+k} = Ru_{-n+k}
\\&
= \sum_{\alpha_1=\cdots=\alpha_{n-k}=0,
\alpha_{n-k+1}\beta_{n-k+1} = \cdots = \alpha_n\beta_n = 0}
(-1)^{\beta_{n-k+1}+\cdots+\beta_n} 
\beta_{n-k+1}!\cdots\beta_n!a_{\alpha,\beta}
\\\quad&
\delta^{(\beta_1)}(x_1)\cdots \delta^{(\beta_{n-k})}(x_{n-k})
x_{n-k+1}^{\alpha_{n-k+1}-\beta_{n-k+1}-1}\cdots
x_{n}^{\alpha_{n}-\beta_{n}-1}
\end{align*}
holds on $\{x \in M \mid x_{n-k+1} > 0,\dots,x_n > 0\} \cap V$ 
with an open neighborhood $V$ of the origin.  
Hence $a_{\alpha,\beta}=0$ holds if 
$\alpha_1=\cdots=\alpha_{n-k}=0$. 
 
In the same way, we conclude that 
$a_{\alpha,\beta}=0$ if the components of $\alpha$ are zero except 
at most $k$ components. 
This implies that $R$ is contained in the left ideal generated by 
$x_{j_1}\cdots x_{j_{k+1}}$ with $1 \leq j_1 < \cdots < j_{k+1} 
\leq n$. 

In the right-hand-side of (\ref{eq:u_-n+k}), 
each term contains the product of at least $n-k$ delta functions.  
Hence $x_{j_1}\cdots x_{j_{k+1}}$ with $1 \leq j_1 < \cdots < j_{k+1} \leq n$, 
and consequently $R$ also, 
annihilates $u_{-n+k}(x)$. 
Hence we have
\[
0 = Pu_{-n+k} = \sum_{i=1}^nQ_i\partial_ix_i u_{-n+k}.
\]
On the other hand, 
since
\[
\partial_ix_i(x_1\cdots x_n)_+^\lambda
= (x_i\partial_i+1)(x_1\cdots x_n)_+^\lambda
= (\lambda+1)(x_1\cdots x_n)_+^\lambda, 
\]
we have
\[
\partial_ix_i u_{-k} = u_{-k-1}
\qquad (k \leq n-1,\,\, 1 \leq i \leq n)
\]
and consequently
\[
0 = \sum_{i=1}^nQ_i\partial_ix_i u_{-n+k} = \sum_{i=1}^n Q_i u_{-n+k-1}. 
\]
By the induction hypothesis, 
$\sum_{i=1}^n Q_i$ belongs to the left ideal of $\Dsc_0$ generated by
\[
x_{j_1}\cdots x_{j_{k}} \quad (1 \leq j_1 < \cdots < j_{k} \leq n), 
\quad
x_1\partial_1 - x_i\partial_i \quad (2 \leq i \leq n). 
\] 
Now rewrite (\ref{eq:Pdivision}) in the form
\begin{align*}
P &= 
\sum_{i=1}^n Q_i \partial_1x_1 
+ \sum_{i=2}^{n} Q_i(\partial_ix_i - \partial_1x_1) 
+ R. 
\end{align*}
If $j_1 > 1$, we have
\[
x_{j_1}\cdots x_{j_k}\partial_1x_1 
= \partial_1x_1x_{j_1}\cdots x_{j_k} .
\]
If $j_1 = 1$, let $l$ be an integer with $2 \leq l \leq n$ 
such that $l \neq j_2,\dots,l \neq j_k$. 
Then we have
\[
x_{j_1}\cdots x_{j_k}\partial_1x_1 
= x_{j_2}\cdots x_{j_k}x_1\partial_1x_1
= x_{j_2}\cdots x_{j_k}x_1(\partial_1x_1 - \partial_lx_l)
+ \partial_l x_{j_2}\cdots x_{j_k}x_1x_l.
\]
We conclude that $P$ belongs to the left ideal generated by
\[
x_{j_1}\cdots x_{j_{k+1}} \quad 
(1 \leq j_1 < \cdots < j_{k+1} \leq n), 
\quad
x_1\partial_1 - x_i\partial_i \quad (2 \leq i \leq n). 
\]
Conversely it is easy to see that these generators 
annihilate $u_{-n+k}$ since 
\[
x_1\partial_1(x_1\cdots x_n)_+^\lambda 
= x_i\partial_i(x_1\cdots x_n)_+^\lambda
=  \lambda(x_1\cdots x_n)_+^\lambda
\]
and each term of (\ref{eq:u_-n+k}) contains the product of 
at least $n-k$ delta functions.  
\end{proof}

\begin{theorem}
Let $f_1,\dots,f_m$ be real-valued real analytic functions defined on a 
neighborhood of the origin of $M = \R^n$ such that 
$df_1\wedge\cdots\wedge df_m \neq 0$. 
Let 
\[
(f_1\cdots f_m)_+^\lambda 
= \sum_{j=-m}^\infty (\lambda+1)^ju_j
\]
be the Laurent expansion about $\lambda=-1$, 
with each $u_j$ being a distribution defined on a common neighborhood of 
the origin.   
Let $v_1,\dots,v_n$ be real analytic vector fields defined on a 
neighborhood of the origin which are linearly independent
and satisfy
\[
v_i(f_j) = \left\{\begin{array}{ll} 1 \quad & (\mbox{if $i=j \leq m$})
   \\ 0 & (\mbox{otherwise})
\end{array}\right.
\]
Then for $k=0,1,\dots,m-1$, the annihilator $\Ann_{\Dsc_0} u_{-m+k}$ 
is generated by 
\begin{align*}&
f_{j_1}\cdots f_{j_{k+1}} 
\quad (1 \leq j_1 < \cdots < j_{k+1} \leq m), 
\\&
f_1v_1 - f_iv_i \quad (2 \leq i \leq m), 
\quad 
v_j \quad (m+1 \leq j \leq n). 
\end{align*}
\end{theorem}

\begin{proof}
By a local coordinate transformation, we may assume that 
$f_j = x_j$ for $j=1,\dots,m$, 
and $v_j = \partial/\partial x_j$ for $j=1,\dots,n$. 
Then the distribution $u_j$ does not depend on $x_{m+1},\dots,x_n$. 
Hence we have only to apply Proposition \ref{prop:main} 
in $\R^m$. 
\end{proof}



\begin{thebibliography}{99}

\bibitem{I}
Igusa,~J., 
\textit{An Introduction to the Theory of Local Zeta Functions},
American Mathematical Society, 2000. 

\bibitem{K}
Kashiwara,~M.,
$B$-functions and holonomic systems---Rationality of roots
              of $B$-functions,
\textit{Invent. Math.},
\textbf{38} (1976), 33--53.

\bibitem{O}
Oaku,~T.,
Annihilators of distributions associated with algebraic local cohomology
of a hypersurface,
\textit{Complex Variables and Elliptic Equations}, 
\textbf{59} (2014), 1533-1546.
\end{thebibliography}
\end{document}